\documentclass[11pt]{amsart}
\usepackage{geometry}                
\usepackage{subfiles}
\usepackage{amssymb}
\usepackage[colorlinks]{hyperref}
\usepackage{graphicx}
\usepackage[normalem]{ulem}
\usepackage{ulem}
\usepackage{epstopdf}
\newtheorem{theorem}{Theorem}[section]
\newtheorem{definition}[theorem]{Definition}

\def\z{{\bf z}}
\def\w{{\bf w}}
\def\P{\mathbb P}
\def\R{\mathbb R}
\def\C{\mathbb C}

\def\H{\mathbb H}

\def\SU{\mathrm{SU}}
\def\PSL{\mathrm{PSL}}

\def\Sp{\mathrm{Sp}}
\def\PSp{\mathrm{PSp}}
\def\HQ{\mathcal{H}_{\mathbb{H}}}

\newtheorem{corollary}[theorem]{Corollary}
\newtheorem{proposition}[theorem]{Proposition}
\newtheorem{lemma}[theorem]{Lemma}

\title[Free and Discrete 2-groups]{Free and Discrete subgroups generated by two quaternionic Heisenberg Translations}
\author{Sagar B.~Kalane and I.D.~Platis}

\address{The Institute of Mathematical Sciences, IV Cross Road, CIT Campus Taramani, Chennai 600 113 Tamil Nadu, India.}
\email{sagarbk@imsc.res.in}
\address{Department of Mathematics, University of Patras, 26b 504 Rio, Achaia, Greece.}
\email{idplatis@upatras.gr}



\date{\today}

\begin{document}
	\begin{abstract}
		Let $A$ and $B$ be two Heisenberg translations of $\Sp(2,1)$ with distinct fixed points. We provide sufficient conditions that guarantee the subgroup $\langle A, B \rangle $ is discrete and free by using Klein's combination theorem. 
	\end{abstract}
	\maketitle
	
	\section{Introduction and statement of results}
	The study of free and discrete two-generated groups within the context of the isometry group $G$ of symmetric spaces of rank 1 of non-compact type is a compelling and active area of research which has evolved significantly over the past few decades. Early research focused on the foundational aspects of group theory and its applications to geometry. The rich structure of symmetric spaces provides a fertile ground for exploring the properties of these groups, which are crucial for understanding various geometric and algebraic phenomena. In the $1980$s, researchers began to systematically investigate the relationships between free groups and discrete subgroups of isometry groups. Explicitly, when $G= \PSL(2,\C)$, Lyndon and Ullman studied this problem for two non-commuting unipotent parabolic elements, see \cite{lyu}.   Many subsequent works have contributed to this subject, see, for instance, \cite{ks},\cite{ig78}. In particular, Riley's influential work notably clarified the links between hyperbolic two-bridge knot and link complements.  Let $\SU(2,1)$ be the isometry group of two-dimensional complex hyperbolic space $\mathbb{H}^2_{\mathbb{C}}.$  Recently, Parker and Kalane established sufficient conditions under which subgroups generated by Heisenberg translations are discrete and free when $G= \SU(2,1);$ see \cite{gk1}.
	In general, the research is ongoing and continues to unveil new dimensions of these groups, influencing both theoretical and applied mathematics.
	
	In this paper, we consider $G= \Sp(2,1)$, which acts isometrically on the two-dimensional quaternionic hyperbolic space $\mathbb{H}^2_{\mathbb{H}}$, where $\mathbb{H}$ is the division ring of quaternions.  We denote by $\HQ$ the quaternionic Heisenberg group, comprising points $p=(\zeta,v)\in \HQ=\mathbb{H}\times\Im(\mathbb{H})$, where $\Im(\mathbb{H})$ is the set of purely imaginary quaternions. Let $p_1=(\zeta_1,v_1)$ and $p_2=(\zeta_2,v_2)$ be points of $\HQ$ and let $B$ and $A$ be the associated Heisenberg left translations respectively, given respectively in matrix form by
	\begin{equation}\label{eq-A-B}
		A = \begin{pmatrix}
			1 & -\sqrt{2}\,\bar{\zeta}_2 & -|\zeta_2|^2 + v_2 \\
			0 & 1 & \sqrt{2}\,\zeta_2 \\
			0 & 0 & 1
		\end{pmatrix}, \quad
		B= \begin{pmatrix}
			1 & 0 & 0 \\
			\sqrt{2}\,\zeta_1 & 1 & 0 \\
			-|\zeta_1|^2 + v_1 & -\sqrt{2}\,\bar{\zeta}_1 & 1
		\end{pmatrix}.
	\end{equation}
	(For the explicit definitions of all the above, see Section \ref{sec-prel}).
	Let $\langle A, B \rangle$ be the group generated by $A$ and $B$. 
	Our goal is to determine conditions on the parameters \(\zeta_j\) and \(v_j\) such that the group \(\langle A, B \rangle\) is discrete and freely generated by \(A\) and \(B\). 
	Let $p=(\zeta, v)\in\HQ$;  the quaternionic Kor\'anyi gauge of $p$ is defined by $ K(p)= \left| |\zeta|^{2} +v \right| ^{1/2}.$ A result that will follow from our main theorem is the following: 
	
	\begin{theorem}\label{thm-main}
		Let $p_i\in\HQ$ be distinct points of the quaternionic Heisenberg group, and let $A$ and $B$ be as in~\eqref{eq-A-B}: here, $A$ fixes $\infty$ and $B$ fixes the origin $o$. If 
		$$
		K(p_1)K(p_2)  \ge 2
		$$
		then the group $\langle A, B \rangle$ is discrete and freely generated by $A$ and $B$.
	\end{theorem}
	Our main theorem slightly stronger:
	
	\begin{theorem}\label{thm-main-improved}
		Let $p_1=(\zeta_1, v_1)$ and $p_2=(\zeta_2, v_2)$ be as in Theorem~\ref{thm-main}, and let $A$ and $B$ be as in~\eqref{eq-A-B}. If any one of the following conditions is satisfied, then the group $\langle A, B \rangle$ is free and discrete:\
		\begin{eqnarray}\label{main-1}
			{K(\zeta_1,v_1)}{K(\zeta_2,v_2)} & \ge & 
			2^{1/2}\left(\left(1-\frac{|v_2|}{{K^2(p_2)}}\right)^{1/3}+\left(1+\frac{|v_2|}{{K^2(p_2)}}\right)^{1/3}\right)^{3/4}\\
			\notag && \quad \times \left(\left(1-\frac{|v_1|}{{K^2(p_1)}}\right)^{1/3}+\left(1+\frac{|v_1|}{{K^2(p_1)}}\right)^{1/3}\right)^{3/4}\\
			\label{main-2}|\zeta_2| K(p_1) 
			&\ge & 
			\frac{2 |\zeta_1|^{2}\Re (\zeta_1 
				\bar{\zeta}_2 )}{|\zeta_2|K^3(p_1)} 
			+ 2,\\
			\label{main-3}|\zeta_1| K(p_2) 
			&\ge& 
			\frac{2 |\zeta_1|^{2}\Re (\zeta_1 
				\bar{\zeta}_2 )}{|\zeta_2|K^3(p_1)} 
			+ 2.
		\end{eqnarray}
	\end{theorem}
	
	By restricting parameters \((\zeta_i, v_i)\) to lie in \(\mathbb{C} \times \mathbb{R}\), we obtain the following corollary, which recovers the main theorem of Parker and Kalane, see \cite{gk1}. Let $(\zeta_j, v_j) = (s_je^{i\theta_j}, t_j)$ for $j=1, 2$. Define ${r_1 = K(\zeta_1,v_1)} = ({s_1}^4+{t_1}^2)^{1/4}$ and ${r_2= K(\zeta_2,v_2)} = ({s_2}^4+{t_2}^2)^{1/4}$.
	
	\begin{corollary}
		\label{prop-main-improved-complex}
		Let \((\zeta_1, v_1)\) and \((\zeta_2, v_2)\) be the elements of the Heisenberg group \(\mathbb{C} \times \mathbb{R}\), and let \(A\) and \(B\) be defined as in~\eqref{eq-A-B}. If any one of the following inequalities is satisfied:
		\begin{enumerate}
			\item[{a)}]
			\begin{eqnarray*} 
				{r_1}^2 {r_2}^2& \ge & 
				2^{1/2}\left(\left(1-\frac{t_2}{{r_2}^2}\right)^{1/3}+\left(1+\frac{t_2}{{r_2}^2}\right)^{1/3}\right)^{3/4} \\
				&&\times \left(\left(1-\frac{t_1}{{r_1}^2}\right)^{1/3}+\left(1+\frac{t_1}{{r_1}^2}\right)^{1/3}\right)^{3/4},
			\end{eqnarray*}
			\item[{b)}] if $s_1\neq0$ then $\begin{displaystyle}
				s_1\,r_2 \ge \begin{cases}
					\begin{displaystyle}2{\left(\frac{~{s_2}}{r_2}\right)}^3\,\cos(\theta_1-\theta_2)+2  \end{displaystyle} & \hbox{if } s_2\neq0, \\
					2 & \hbox{if } s_2=0.
				\end{cases}
			\end{displaystyle}$
			\item[{c)}] if $s_2\neq 0$ then $\begin{displaystyle}
				s_2\,r_1  \ge \begin{cases}
					\begin{displaystyle}2{\left(\frac{~{s_1}}{r_1}\right)}^3\,\cos(\theta_1-\theta_2)+2  \end{displaystyle} & \hbox{if } s_2\neq0, \\
					2 & \hbox{if } s_2=0.
				\end{cases}
			\end{displaystyle}$
		\end{enumerate}
		then the group \(\langle A, B \rangle\) is discrete and free.
	\end{corollary}
	Moreover, in the special case of vertical Heisenberg translations, our result not only recovers but also unifies and extends the independent results of Kalane–Tiwari and Parker, see in \cite{sd}, \cite{par}.

	The paper is organized as follows: in Section \ref{sec-prel} we review all notions that we use throughout this work. In Section \ref{sec-relation} we prove a relation between Cygan spheres that is crucial for the proof of our main theorem. The proofs of the computational lemmas required are given in Section \ref{sec-appendix}. Finally, in Section \ref{sec-main} we prove Theorem \ref{thm-main-improved}.
	
	\section{Preliminaries}\label{sec-prel}
	In Section \ref{sec-quat}, we present the basics about quaternions and in Section\ref{sec-def}, we present quaternionic hyperbolic plane and its group of isometries with respect to its quaternion K\"ahler structure. Quaternionic Heisenberg group and its group of similarities are found in Section \ref{sec-qheis}. We close this section with the definition of Kor\'anyi spherical coordinates (Section \ref{sec-koranyi}) and an estimate for the Cygan distance (Section \ref{sec-estimate}).
	
	\subsection{Quaternions}\label{sec-quat} Let $\mathbb{H}$ denote the division ring of quaternions. Any element $a \in \mathbb{H}$ can be uniquely expressed as $a= a_0 +a_1 i +a_2 j +a_3 k$, where $a_0, a_1, a_2, a_3 \in \mathbb{\R}$ and $i,j,k$ satisfy the relations $i^2= j^2=k^2=ijk= -1$. Let $\bar a = a_0 -a_1 i -a_2 j -a_3 k$ be the conjugate of $a$. The real part $\Re(a)$ and the imaginary part $\Im(a)$ of $a$ are respectively defined by $\Re(a) = a_0$ and $\Im(a)=a_1 i +a_2 j +a_3 k$. The modulus of $a$ is defined by $|a| = \sqrt{a\bar a} =\sqrt{{a_0}^2+{a_1}^2+{a_2}^2+{a_3}^2}.$ Two quaternions $a$ and $b$ are similar if there exists a nonzero $\mu \in\mathbb{H}$ such that $a = \mu b {\mu}^{-1}$. Thus, $a$ and $b$ are similar if and only if $\Re(a)=\Re(b)$ and $|a|=|b|.$
	
	A quaternion $q$ is a unit quaternion if $|q|=1$; it can be written as $q= \cos\theta + \mu \sin\theta,$ where $\theta \in [0,\pi],$ and $\mu$ is a purely imaginary unit quaternion satisfying $\mu^2=-1$. The set of such quaternions is in bijection with the unit sphere $S^2$. 
	If $a$ is any quaternion with modulus $r$, then it can be written as $a = r (\cos\theta + \mu \sin\theta)$,
	where  $\mu$ is a purely imaginary unit quaternion.
	\subsection{Quaternionic Hyperbolic plane}\label{sec-def}
	Let $\mathbb{H}^{2,1}$ denote the three-dimensional right vector space over $\mathbb{H}$, equipped with the Hermitian form of signature $(2,1)$, defined by:
	$$
	\langle\z,\w\rangle=\w^{\ast}H\z=\bar w_3 z_1+\bar w_2 z_2+\bar w_1 z_3,
	$$
	where $\z,\,\w$ are column vectors in $\H^3$ and the matrix of the Hermitian form is given by
	\begin{center}
		$H = \begin{pmatrix}
			0 & 0 & 1\\
			0 & 1 & 0 \\
			1 & 0 & 0\\
		\end{pmatrix}.$
	\end{center}
	For any ${\bf w} \in \mathbb{H}^{2,1}$,  $\langle\w,\w\rangle\in\R$. Based on whether this value is positive, null or negative, we can divide $\H^{2,1}\setminus \{\bf 0\}$ into three disjoint sets: \begin{eqnarray*}
		V_+ & = & \{\w\in\H^{2,1}:\langle\w,\w \rangle>0\}, \\
		V_{0} & = & \{\w \in\H^{2,1}\setminus \{{\bf 0}\}:\langle\w,\w \rangle=0\},\\
		V_{-} & = & \{\w\in\H^{2,1}:\langle\w,\w \rangle<0\}.
	\end{eqnarray*}
	The projective model of quaternionic hyperbolic plane is defined by $\mathbb{H}^2_{\mathbb{H}}=\P (V_{-})$, where $\P$ is the natural projection map on $\H^{2,1}\setminus \{{\bf 0}\}.$ 
	$\mathbb{H}^2_{\mathbb{H}}$ is a quaternion-K\"ahler manifold with its metric defined by
	$$
	ds^2=-\frac{4}{\langle w,w\rangle^2}\det\left(\begin{matrix}
		\langle \w,\w\rangle &\langle d\w,\w\rangle \\
		\langle \w,d\w\rangle &\langle d\w,d\w\rangle\end{matrix}\right).
	$$
	Equivalently, it is given by the distance function $\rho$ where
	$$
	\cosh^2\left(\frac{\rho(\P \w_1,\P \w_2)}{2}\right)=\frac{\langle\w_1,\w_2\rangle \langle\w_2,\w_1\rangle}{\langle\w_1,\w_1\rangle\langle\w_2,\w_2\rangle}.
	$$
	The boundary of quaternionic hyperbolic plane is defined by $\partial \mathbb{H}^2_{\mathbb{H}}=\P (V_{0}).$ Let $\Sp(2,1)$ be the group of $3 \times 3$ quaternionic matrices that preserves the above Hermitian form; it acts isometrically on $\mathbb{H}^2_{\mathbb{H}}$ and the full isometry group of 
	$\mathbb{H}^2_{\mathbb{H}}$ is the projective unitary group $\PSp(2,1) = \Sp(2,1)/ \{I, -I\}.$
	Isometries of $\PSp(2,1)$ are classified according to their fixed points. An isometry in $\PSp(2,1)$ is called \emph{hyperbolic} (resp. \emph{parabolic}) if it has exactly two fixed points (resp. one fixed point) on the boundary $\partial \mathbb{H}^2_{\mathbb{H}}$. It is called \emph{elliptic} if it fixes a point inside $\mathbb{H}^2_{\mathbb{H}}.$

	The Siegel domain model of quaternionic hyperbolic plane is obtained by taking the section $w_3=1$; then 
	$$ \mathbb{H}^2_{\mathbb{H}} =\mathcal{S}=\{ (w_1,w_2) \in {\H}^2:  2\Re(w_1)+|w_2|^2<0\}  .$$ The lift of $\infty$ is defined by $ q_\infty =(1,0,0)^T,$ so that $\infty \in \partial \mathbb{H}^2_{\mathbb{H}}$. 
	There exists a bijection between the set $\P(V_0)\setminus\{\infty\}$ and the set
	$$
	\partial\mathcal{S}=\{(w_1,w_2)\in\mathbb{H}^2:\,2\Re(w_1)+|w_2|^2=0\}.
	$$
	The latter is in bijection with the set 
	$\mathbb{H} \times \Im(\mathbb{H})$:
	$$
	\mathbb{H} \times \Im(\mathbb{H})\ni(w,s)\mapsto (-|w|^2+s,w/2)\in \partial\mathcal{S}. 
	$$
	Thus, the natural identification
	of the set of the finite boundary points of quaternionic hyperbolic plane to the set $\mathbb{H} \times \Im(\mathbb{H})$. Hence, the full boundary $\partial \mathbb{H}^2_{\mathbb{H}} $ can be viewed as the one-point compactification of $\mathbb{H} \times \Im(\mathbb{H}).$

	\subsection{Quaternionic Heisenberg group}\label{sec-qheis} The {\it quaternionic Heisenberg group} $\HQ$ is $\HQ=\mathbb{H}\times\Im(\mathbb{H})=\partial \mathbf{H}^2_{\mathbb{H}}\setminus\{\infty\}$ with group multiplication given by
	$$
	(\zeta_1, v_1)(\zeta_2, v_2) = (\zeta_1 + \zeta_2, v_1 + v_2 + 2\Im(\bar{\zeta_2} \zeta_1)),
	$$
	for each $(\zeta_1,v_1)$ and $(\zeta_2,v_2)\in \HQ$. The Cygan metric is defined as follows. Consider the map $\kappa:\HQ\to\mathbb{H}$ given by $\kappa(\zeta,v)=-|\zeta|^2+v$ for each $(\zeta,v)\in\HQ$. The {\it quaternionic Kor\'anyi gauge}  is then given by   $$ K(\zeta,v)= 
	\big| \kappa(\zeta,v) \big| ^{1/2} = \big(|\zeta|^{4} +|v|^2 \big) ^{1/4}.$$
	Then the {\it Cygan metric} on the quaternionic Heisenberg group $\HQ$ is given by
	$$
	\rho_0\big((\zeta_1,v_1),\, (\zeta_2,v_2)\big) = K\left((\zeta_2, v_2)^{-1} (\zeta_1, v_1)\right)= \big| |\zeta_1 - \zeta_2|^2 + v_1 - v_2 - 2 \Im(\bar{\zeta}_2 \zeta_1 ) \big|^{\frac{1}{2}},
	$$
	for each $(\zeta_i,v_i)\in\HQ$, $i=1,2$. This distance is invariant under the following transformations:
	
	{\it Translations}: The group $\HQ$ acts on itself via left-translations. For a fixed $p_0=(\zeta_0, v_0) \in \HQ$, the translation map $T_{p_0}$ is given by  
	\[
	T_{p_0} : p=(\zeta, v) \mapsto p_0p=(\zeta_0 + \zeta, v_0 + v + 2\Im(\bar{\zeta}\zeta_0)).  
	\]  
	As a matrix in $\Sp(2,1)$, 
	the transformation $T_{p_0}$ is given by  
	$$
	T_{p_0} =  
	\begin{pmatrix}  
		1 & -\sqrt{2} \bar{\zeta_0} & \kappa(p_0) \\  
		0 & 1 & \sqrt{2} \zeta_0 \\  
		0 & 0 & 1  
	\end{pmatrix}. 
	$$
	A translation is a parabolic transformation that fixes $\infty$. When $\zeta_0 = 0$, the transformation $T_{(0, v_0)}$ is called a \emph{vertical translation}; otherwise, it is called a \emph{non-vertical translation}.
	
	\smallskip

	{\it Rotations:} The group $\Sp(1)$ acts on $\HQ$ by rotations. A rotation $R_\mu$, $\mu\in\Sp(1)$, is given for each $(\zeta,v)\in\HQ$ by 
	\[ R_\mu: (\zeta, v) \mapsto (\mu \zeta, v ), ~\mu \in \Sp(1).  
	\] 
	As a matrix in $\Sp(2,1)$, $R_\mu$ given by  
	$$
	R_\mu =  
	\begin{pmatrix}  
		1 & 0 &0 \\  
		0 & \mu & 0 \\  
		0 & 0 & 1  
	\end{pmatrix}.  
	$$
	
	{\it Dilations:} The set of positive real numbers $\mathbb{R}^+$ acts on $\HQ$ by dilations. A dilation $D_\delta$, $\delta>0$, is defined for each $(\zeta,v)\in\HQ$ by 
	\[ D_\delta:(\zeta, v) \mapsto (\delta\zeta, \delta^2v ). 
	\]
	As a matrix in $\Sp(2,1)$,
	$$
	D_\delta=\begin{pmatrix}  
		\delta & 0 &0 \\  
		0 & 1 & 0 \\  
		0 & 0 & 1/\delta  
	\end{pmatrix}.  
	$$
	All the above transformations extend to transformations of $\PSp(2,1)$ which fix infinity. In fact, the {similarity group}, comprising translations, dilations and rotations is the full subgroup of $\PSp(2,1)$ which fixes $\infty.$ 
	
	Inversion $\iota$ is the involution 
	that exchanges origin $o=(0,0)$ and $\infty$. In Heisenberg coordinates, it is defined by 
	$$ \iota:(\zeta, v) \mapsto \left(\frac{2\zeta}{-|\zeta|^2+v} , \frac{-4v}{|\zeta|^4+|v|^2}\right), 
	$$
	and in  matrix form, invwrsion is given by
	\begin{equation}\label{eq-iota}
		\iota=\left(\begin{matrix} 0 & 0 & 1 \\ 0 & 1 & 0 \\ 1 & 0 & 0 \end{matrix}\right).
	\end{equation}
	
	Every element of ${\rm Sp}(2,1)$ is a composition of similarities and inversion $\iota$. For more details, see in \cite{kp1}.

	\subsubsection{Cygan spheres}
	For a given center $p_0 \in \partial \mathbb{H}^2_{\mathbb{H}}$ and a given radius $r > 0$, the {Cygan sphere} of radius $r$ centered at $p_0$ is given by
	\begin{equation}
		S_r(p_0) = \{ p = (\zeta,v) \in \HQ \setminus\{\infty\}: \rho_0(p, p_0) = r \}.
	\end{equation}
	In particular, when the center is at the origin $o = (0,0)$, the Cygan sphere is
	\begin{equation}
		S_r(0) = \{ p = (\zeta,v) \in \HQ : |\zeta|^4 + |v|^2 = r^4 \}.
	\end{equation}
	\subsection{Quaternionic Kor\'anyi coordinates}\label{sec-koranyi}
	Every $(\zeta,v)\in\HQ$ 
	has the following representation:
	$$
	(\zeta,v)=\left(r\sqrt{\cos\psi}~U,\,r^2\sin\psi~ u_2\right),
	$$
	where
	$$
	r=K(\zeta,v)=||\zeta|^2+v|^{1/2}\ge 0,\quad  \psi=\arccos\left(\frac{|\zeta|^2}{r^2}\right)\in[-\pi/2,\pi/2],
	$$
	and $U=\cos\varphi+\sin\varphi\,u_1$ is a unit quaternion such that $\varphi\in[0,2\pi]$ and $u_i$, $i=1,2$ are purely imaginary quaternions:
	\begin{eqnarray*}
		u_1&=&\cos\alpha_1\cos\alpha_2~i+\cos\alpha_1\sin\alpha_2~j+\sin\alpha_1~k,\\
		u_2&=&\cos\beta_1\cos\beta_2~i+\cos\beta_1\sin\beta_2~j+\sin\beta_1~k,\
	\end{eqnarray*}
	such that
	$ \alpha_1, \beta_1 \in [-\pi/2,\pi/2],$ $ \alpha_2, \beta_2 \in [0,2 \pi]. 
	$
	\subsection{An estimate for the Cygan distance}\label{sec-estimate}
	Let $p_i=(\zeta_i,v_i)\in \HQ$. Then their Cygan distance is given by
	\begin{eqnarray*}
		\rho_0^2(p_1,p_2)&=&\left| |\zeta_1 - \zeta_2|^2 + v_1 - v_2 - 2 \Im(\overline{\zeta_2} \zeta_1) \right|\\
		&=&\left||\zeta_1|^2+|\zeta_2|^2-2\overline{\zeta_2}\zeta_1
		+v_1-v_2\right|\\
		&=&\left||\zeta_1|^2+|\zeta_2|^2+v_1-v_2-2\overline{\zeta_2}\zeta_1
		\right|\\
		&\le&\left||\zeta_1|^2+v_1+|\zeta_2|^2-v_2\right|+2|\overline{\zeta_2}\zeta_1|\\
		&=&\left((|\zeta_1|^2+|\zeta_2|^2)^2+|v_1-v_2|^2\right)^{1/2}+2|\zeta_2||\zeta_1|,
	\end{eqnarray*}
	where we have used the triangle inequality.
	Consider the Kor\'anyi coordinates
	$$
	(r_i\sqrt{\cos\psi_i}U_i,\,r_i^2\sin\psi_i~u_{2,i})
	$$
	for the points $p_i$, $i=1,2$, respectively. Then the above estimate can be expressed as
	\begin{eqnarray}\label{eq-gen-est}
		d=\rho_0^2(p_1,p_2)&\le &\sqrt{r_1^4+r_2^4+2r_1^2r_2^2(\cos\psi_1\cos\psi_2+\sin\psi_1\sin\psi_2\cos\phi)}\\
		\notag&&+2r_1r_2\sqrt{\cos\psi_1\cos\psi_2},    
	\end{eqnarray}
	where $\phi$ is the angle between the unit quaternions $u_{2,1}$ and $u_{2,2}$.
	In particular, when $r_1=r_2=r$, the estimate becomes
	\begin{eqnarray}\label{eq-sph-est}
		d/r^2&\le &\sqrt{2(1+\cos\psi_1\cos\psi_2+\sin\psi_1\sin\psi_2\cos\phi)}\\
		\notag&& + 2\sqrt{\cos\psi_1\cos\psi_2}. 
	\end{eqnarray}
	Inequality (\ref{eq-gen-est}) becomes an equality if and only if there exists a non-positive $\lambda$ such that
	$$
	|\zeta_1|^2+|\zeta_2|^2+v_1-v_2=2\lambda \overline{\zeta_2}\zeta_1.
	$$
	In the case where $\lambda=0$, we necessarily have $\zeta_1=\zeta_2=0$ and $ v_1=v_2$, which is impossible since the points are assumed to be distinct. Thus $\lambda <0$, and
	\begin{eqnarray*}
		&&
		|\zeta_1|^2+|\zeta_2|^2=2\lambda\Re(\overline{\zeta_2}\zeta_1),\\
		&&
		v_1-v_2=2\lambda\Im(\overline{\zeta_2}\zeta_1).
	\end{eqnarray*}
	If $v_1=v_2$, then $\Im(\overline{\zeta_2}\zeta_1)=0$, and by the first equation, we also obtain $\Re(\overline{\zeta_2}\zeta_1)<0$. This implies that $\zeta_1 = k \zeta_2$ for some real constant $k<0$. 
	Furthermore, if $(\zeta_i,v_i)$, $i=1,2$, lie in the Cygan sphere of radius $r$ with $v_1=v_2$, then
	$$
	|\zeta_i|^4+v_i^2=r^4, \quad i=1,2,
	$$
	and subtracting these two equations gives $|\zeta_1|=|\zeta_2|$. Since $\zeta_1 = k \zeta_2$ for some $k<0$, we must have  $\zeta_1 = - \zeta_2.$
	
	When $v_1\neq v_2$, we have
	\begin{equation*}\label{eq-est}
		\lambda=\frac{|\zeta_1|^2+|\zeta_2|^2}{2\Re(\overline{\zeta_2}\zeta_1)},
	\end{equation*}
	hence
	\begin{equation}\label{eq-ineq}
		\frac{\Im{(\overline{\zeta_2}\zeta_1)}}{\Re(\overline{\zeta_2}\zeta_1)}=\frac{v_1-v_2}{|\zeta_1|^2+|\zeta_2|^2}.   
	\end{equation}
	Assuming that the points lie on the same Cygan sphere and using quaternionic Kor\'anyi coordinates, Equation (\ref{eq-ineq}) becomes
	\begin{equation}\label{eq-ineq-sph}
		\tan\Phi=\frac{\sin\psi_1~u_{2,1}-\sin\psi_2~u_{2,2}}{\cos\psi_1+\cos\psi_2},    
	\end{equation}
	where
	$$
	\Phi=\arctan\left(\frac{\Im({\overline U_2}U_1)}{\Re({\overline U_2}U_1)}
	\right).
	$$
	Equation (\ref{eq-ineq-sph}) between imaginary quaternions is equivalent to the following system of three real equations for the variables $\Phi$, $\psi_i$, and $\beta_{1,i},\beta_{2,i}$, $i=1,2$:
	\begin{eqnarray*}
		\sin\psi_1\cos\beta_{1,1}\cos\beta_{2,1}-\sin\psi_2\cos\beta_{1,2}\cos\beta_{2,2}&=&\tan\Phi(\cos\psi_1+\cos\psi_2),\\
		\sin\psi_1\cos\beta_{1,1}\sin\beta_{2,1}-\sin\psi_2\cos\beta_{1,2}\sin\beta_{2,2}&=&0,\\
		\sin\psi_1\sin\beta_{1,1}-\sin\psi_2\sin\beta_{1,2}&=&0.
	\end{eqnarray*}

	\section{A relation between Cygan spheres}\label{sec-relation}
	In this section, we prove the following proposition:
	\begin{proposition}\label{prop-diam}
		Let 
		$p_1=(\zeta_1,v_1)=\left( r\sqrt{\cos\psi_1} U_1,\, r^2 \sin\psi_1 u_{2,1} \right)$ 
		be a point on the Cygan sphere $S_r(0, 0)$. Then the Cygan sphere $S_r(0, 0)$ is entirely contained in the Cygan sphere $S_{d_r}(p_1)$, where $$d_r \geq 2^{1/4}r \left( \left(1 - \sin \psi_1 \right)^{1/3} + \left(1 + \sin \psi_1 \right)^{1/3} \right)^{3/4}.$$ 
	\end{proposition}
	To establish this result, we need the following lemmas, whose proofs are found in the Appendix.
	\begin{lemma}\label{lem-trig}
		Let $\alpha\in [-\pi/4,\pi/4]$, and define the function $f_\alpha:I=[-\pi/4,\pi/4]\longrightarrow {\mathbb R}$ by
		$$
		f_\alpha(\theta)=\cos(\theta+\alpha)+\sqrt{\cos(2\theta)\cos(2\alpha)},\quad \theta\in I.$$
		Then the maximum of $f_\alpha$ on this interval is given by
		$$\max_{\theta\in I}f_\alpha(\theta) =\frac{1}{\sqrt{2}}\Bigl(\bigl(1-\sin(2\alpha)\bigr)^{1/3}+\bigl(1+\sin(2\alpha)\bigr)^{1/3}\Bigr)^{3/2}.$$
	\end{lemma}
	\begin{lemma}\label{mam-bound}
		For $\psi_1 \in \left[-\frac{\pi}{2}, \frac{\pi}{2}\right]$, define the function $h$ given by
		$$h(\psi_1) =  \sqrt{2} \left( \left(1 - \sin \psi_1 \right)^{1/3} + \left(1 + \sin \psi_1 \right)^{1/3} \right)^{3/2}.$$ Then the minimum value of $h(\psi_1)$ occurs at $\psi_1=\pm \frac{\pi}{2}$ and it is equal to $2$. The maximum value occurs at $\psi_1=0$ and is equal to $4$. 
	\end{lemma}
	\begin{lemma}\label{lem-bound}
		For given $\psi_1 \in \left[-\frac{\pi}{2}, \frac{\pi}{2}\right]$, define the function
		$$\eta_{\psi_1}(\psi_2, \phi) : A=\left[-\frac{\pi}{2}, \frac{\pi}{2}\right] \times [0, \pi] \to \mathbb{R},$$  
		by
		\begin{align*}
			\eta_{\psi_1}(\psi_2, \phi) &= \sqrt {2(1+ \cos \psi_1\cos \psi_2 + \cos \phi \sin \psi_1 \sin \psi_2 )  }+ 2 \sqrt{\cos\psi_2 \cos\psi_1},
		\end{align*}
		for all $(\psi_2,\phi)\in A.$
		Then the maximum value of $\eta_{\psi_1}(\psi_2, \phi)$ over $A$ is given by
		$$\max_{(\psi_2, \phi)\in A} \eta_{\psi_1}(\psi_2, \phi) =  \sqrt{2} \left( \left(1 - \sin \psi_1 \right)^{1/3} + \left(1 + \sin \psi_1 \right)^{1/3} \right)^{3/2}.$$
	\end{lemma}
	
	\noindent{\it Proof of Proposition \ref{prop-diam}.}
	
	Let $p_2=(\zeta_2, v_2)$ be any point on the Cygan sphere $S_r(0, 0)$. From the estimate (\ref{eq-sph-est}), we obtain that 
	\begin{eqnarray*}
		\rho_0(p_1,p_2)
		& \le & r\left(   \sqrt {2(1+ \cos \psi_1~ \cos \psi_2 + \cos \phi ~\sin \psi_1 ~\sin \psi_2 )  }+ 2 \sqrt{\cos\psi_2 \cos\psi_1} \right)^{\frac{1}{2}},\\
	\end{eqnarray*}
	where $\cos \phi=u_{2,1} . u_{2,2} $.
	From Lemma \ref{lem-bound} we subsequently deduce
	\begin{eqnarray*}
		\rho_0(p_1,p_2) & \le &  {2^{1/4}} r\left( \left(1 - \sin \psi_1 \right)^{1/3} + \left(1 + \sin \psi_1 \right)^{1/3} \right)^{3/4}.
	\end{eqnarray*}
	\qed
	
	\section{Proof of Theorem \ref{thm-main-improved}}\label{sec-main}
	We first give the definition of isometric spheres in the context of quaternionic hyperbolic geometry in Section \ref{sec-isometric}. The proof of relation (\ref{main-1}) of Theorem \ref{thm-main-improved} is given here. Quaternionic fans and a fundamental domain bounded by a quaternionic fan are discussed in Section \ref{sec-fan}, and relations (\ref{main-2}) and (\ref{main-3}) are proved in Proposition \ref{prop-fan-domain}.
	\subsection{Isometric spheres}\label{sec-isometric}
	In \cite{gol}, Goldman extended the concept of isometric spheres from real hyperbolic geometry to the complex hyperbolic case. Later, Parker and Cao defined isometric spheres in the setting of quaternionic hyperbolic geometry, see in \cite{wp1}. Let $P\in\Sp(2,1)$ and let $P^{-1}$ be its inverse. Suppose 
	\begin{equation}\label{eq P}
		P=\left(\begin{matrix} a & b & c \\ d & e &f  \\ g & h & j \end{matrix}\right),
		\qquad
		P^{-1}=\left(\begin{matrix} \bar j & \bar f & \bar c \\ \bar h & \bar e &\bar b  \\ \bar g & \bar d & \bar a \end{matrix}\right).
	\end{equation}
	If $g\neq 0$ then $P$ does not fix $\infty$. Moreover, $P(\infty)$ and $P^{-1}(\infty)$ lies in $\partial \mathbb{H}^2_{\mathbb{H}} \setminus\{\infty\}$ have quaternionic Heisenberg coordinates $$ P(\infty)= \left(dg^{-1}/{\sqrt 2},\Im(ag^{-1}) \right), \qquad
	P^{-1}(\infty)= \left(\bar h {\bar g}^{-1}/{\sqrt 2},-\Im(j{\bar g}^{-1}) \right).
	$$
	\begin{definition}
		The {\it isometric sphere $I_P$ corresponding to an element $P$} in $\Sp(2,1)$ with $P(\infty)\neq \infty$ is the Cygan sphere with center $P^{-1}(\infty)$, and radius $r_P=1/\sqrt{|g|}$. Similarly, the {\it isometric sphere $I_{P^{-1}}$  } corresponding to $P^{-1}$ is
		the Cygan sphere with center $P(\infty)$ and radius $r_{P^{-1}}=r_P=1/\sqrt{|g|}$.  
	\end{definition}
	The involution $\iota$, defined by Eq.(\ref{eq-iota}) which interchanges $o$ and $\infty$, maps the Cygan sphere with center $o$ and radius $r$ to the Cygan sphere with center $o$ and radius $1/r$.
	Let $A$ and $B$, as defined in \eqref{eq-A-B}, be the associated Heisenberg translations fixing $\infty$ and $o=(0,0)$ respectively.
	The Cygan isometric sphere of $B$ has radius $r_B=1/K(p_1)$, $p_1=(\zeta_1,v_1)$, and center 
	$$
	B^{-1}(\infty)=\left( -\zeta_1(\kappa(p_1^{-1}))^{-1},
	\,  \frac{v_1}{K^4(p_1)}\right).$$
	Similarly, the isometric sphere of $B^{-1}$ has radius $r_{B^{-1}}=r_B=1/K(p_1)$ and center
	$$
	B(\infty)=\left( \zeta_1(\kappa(p_1))^{-1},\  \frac{-v_1}{K^4(p_1)}\right).
	$$ 
	Note that the origin $o=(0,0)$, which is the fixed point of $B$, lies on both of these isometric spheres and they are both tangent at this point. A fundamental domain for the action of the cyclic group $\langle B\rangle$ on the Heisenberg group can thus be described as the region outside both of these isometric spheres.
	\begin{proposition}\label{prop-iso-diam}
		The isometric spheres of $B$ and $B^{-1}$ are contained in the Cygan ball with center $o$ and radius $R_o$ where
		$$
		R_o=\frac{2^{1/4}}{K(p_1)}
		\left(\left(1-\frac{|v_1|}{K^2(p_1)}\right)^{1/3}+\left(1+\frac{|v_1|}{K^2(p_1)}\right)^{1/3}\right)^{3/4}.
		$$
	\end{proposition}
	\begin{proof}
		The Cygan distance is invariant under Heisenberg translations. Therefore, translating by $p_1^{-1}$ we have a Cygan sphere passing through the origin, with center $p_1^{-1}$ and radius $r$. 
		
		The isometric spheres of $B$ and $B^{-1}$ are examples of such spheres with radius
		$r_{B^{-1}}=r_B=K(p_1)^{-1}$.
		Now, by comparing the center of $B$ with $p_1^{-1}$, we get 
		$$
		r_B\sqrt{\cos(\psi_1)} (\cos\theta_1+u_1 \sin\theta_1)={\zeta_1} (\kappa(p_1))^{-1}.
		$$
		From this, we deduce
		$$\cos(\psi_1)=\frac{| \zeta_1|^2}{K^2(p_1)}, \quad
		\sin(\psi_1)=\frac{|v_1|}{K^2(p_1)}.
		$$
		By applying Proposition \ref{prop-diam}, we obtain an upper bound for the Cygan distance between points $p_i=(\zeta_i,v_i)\in S_r(0,0)$, $i=1,2$:
		\begin{eqnarray*}
			\rho_0(p_1,p_2) & \le & {2^{1/4}}r \left( \left(1 - \frac{|v_1|}{K^2(p_1)} \right)^{1/3} + \left(1 + \frac{|v_1|}{K^2(p_1)} \right)^{1/3} \right)^{3/4}.
		\end{eqnarray*}
		Therefore, the isometric spheres of $B$ is contained in the Cygan ball with center $o$ and radius $R_o$, where
		\begin{eqnarray*}
			R_o & = & \frac{2^{1/4}}{K(p_1)}
			\left(\left(1-\frac{|v_1|}{K^2(p_1)}\right)^{1/3}+\left(1+\frac{|v_1|}{K^2(p_1)}\right)^{1/3}\right)^{3/4},
		\end{eqnarray*}
		and the same is true for the isometric sphere of $B^{-1}.$ This completes the proof.
	\end{proof}
	We will use the following version of Klein's combination theorem on the boundary of quaternionic hyperbolic space. 
	\begin{proposition}{\label{KCT}}
		Let $A$ and $B$ be the Heisenberg translations fixing $\infty$ and $o$ respectively, given by \eqref{eq-A-B}. Suppose $D_A\subset \partial \mathbb{H}^2_{\mathbb{H}}$ and $D_B\subset \partial \mathbb{H}^2_{\mathbb{H}}$ are fundamental domains for $\langle A\rangle$ and $\langle B\rangle$, respectively. If the interiors of the domains intersect $D_A^\circ\cap D_B^\circ\neq \emptyset$, and their closures satisfy $\overline{D}_A\cup\overline{D}_B=\partial \mathbb{H}^2_{\mathbb{H}}$, then the subgroup $\langle A,B \rangle$ is free and discrete.
	\end{proposition}
	
	\noindent{\it Proof of Theorem \ref{thm-main-improved}.}
	Let $A$ be the Heisenberg translation by $p_2=(\zeta_2,v_2)$ and consider
	the involution $\iota$ given in \eqref{eq-iota}; recall that $\iota$ maps the Cygan sphere centered at $o$ of radius $R$ to the same Cygan sphere
	centered at $o$ with radius $1/R$.
	
	Consider next the element $\iota A\iota$. It has the same form as $B$, but with subscript $1$ consistently replaced by subscript $2$. In particular, 
	the isometric spheres of $\iota A\iota$ and its inverse both have radius $1/K(p_2)$ and center 
	$$
	\iota A^{-1}\iota(\infty)=\left( -\zeta_2(\kappa(p_2^{-1})),\  \frac{v_2}{K^4(p_2)}\right)$$
	$$\iota A\iota(\infty)=\left( \zeta_2(\kappa(p_2)),\  \frac{-v_2}{K^4(p_2)}\right).
	$$ 
	Using Proposition~\ref{prop-iso-diam}, we conclude that they are contained in the Cygan ball with center $o$ and radius $R_\infty$ where
	$$
	R_\infty=\frac{2^{1/4}}{K(p_2)}
	\left(\left(1-\frac{|v_2|}{K^2(p_2)}\right)^{1/3}+\left(1+\frac{|v_2|}{K^2(p_2)}\right)^{1/3}\right)^{3/4}.
	$$
	Therefore, the image under $\iota$ of these spheres is contained in the exterior of a Cygan sphere with center $o$ and radius
	$1/R_\infty$. Hence, if $R_o\le 1/ R_\infty$ then we can use the Klein combination theorem to conclude that $A$ and $B$ 
	freely generate the subgroup $\langle A,\,B\rangle$. The inequality $R_o\le 1/ R_\infty$ is equivalent to
	\begin{eqnarray*}
		1\ge R_\infty R_o &=& \frac{2^{1/2}}{K(p_2)K(p_1)}
		\left(\left(1-\frac{|v_2|}{K^2(p_2)}\right)^{1/3}+\left(1+\frac{|v_2|}{K^2(p_2)}\right)^{1/3}\right)^{3/4}\\
		&\times& \left(\left(1-\frac{|v_1|}{K^2(p_1)}\right)^{1/3}+\left(1+\frac{|v_1|}{K^2(p_1)}\right)^{1/3}\right)^{3/4}.
	\end{eqnarray*}
	Multiplying both sides by $K(p_2)K(p_1)$, we obtain condition (1') of Theorem~\ref{thm-main-improved}.
	
	\subsection{Fundamental domain bounded by a fan}\label{sec-fan}
	Fans are a special class of hypersurfaces in the quaternionic Heisenberg group $\HQ$, as discussed in \cite{JL}. An infinite fan is a copy of the six-dimensional Euclidean space $\mathbb{R}^6$ within $\HQ$, whose image under the vertical projection is a real affine hyperplane in the quaternions $\mathbb{H}$.
	
	Fix a $q_0= |q_0| u_0\in{\mathbb H}$, where $u_0$ is a unit quaternion. The {\it fan of $q_0$ with vertex at infinity} is given in Heisenberg coordinates by
	$$
	F^{(\infty)}_{|q_0|}=\{(q,w) \in \HQ:\,   \Re(q\bar u_0)= |q_0|\},
	$$
	so that the standard lift of a point $(q,w)\in F^{(\infty)}_{|q_0|}$ is  
	$$
	\left(\begin{matrix} -|q_0|^2-|v|^2+ w  \\ 
		\sqrt{2}(|q_0|+v)u_0 \\ 1 \end{matrix}\right),\quad v=\Im(q\overline{u_0}).
	$$
	Recall the vertical projection $\Pi:\HQ\to \mathbb{H}$: 
	$\Pi(\zeta,v)=\zeta$ for each $(\zeta,v)\in\HQ$.
	We immediately have
	\begin{lemma}\label{lem-project-fan-infty}
		The image of $F^{(\infty)}_{|q_0|}$ under the vertical projection is the real affine hyperplane
		$$
		F=\left\{q \in \mathbb{H}:  \Re(q\bar u_0)= |q_0| \right\}
		$$
		and thus $F^{(\infty)}_{|q_0|}= \Pi^{-1}(F)$.
	\end{lemma}
	Let $A$ be the Heisenberg translation by $p_2=(\zeta_2,v_2)$  (\eqref{eq-A-B}). Consider the infinite fan $F^{(\infty)}_{|q_0|}$, where $q_0= |q_0| u_0\in{\mathbb H}$ and $u_0=\frac{\bar \zeta_2}{|\zeta_2|}$. The map $A$ maps 
	$F^{(\infty)}_{|q_0|}$ to $F^{(\infty)}_{|q_0| + |\zeta_2|}$:
	\begin{eqnarray*}
		\left(\begin{matrix} 
			1 & -\sqrt{2}\,\bar{\zeta}_2 & -|\zeta_2|^2 + v_2 \\
			0 & 1 & \sqrt{2}\,\zeta_2 \\
			0 & 0 & 1
		\end{matrix}\right)
		\left(\begin{matrix}-|q_0|^2-|v|^2+ w  \\ 
			\sqrt{2}(|q_0|+v)u_0 \\ 1 \end{matrix}\right)&=&\\
		\left(\begin{matrix} -(|q_0|+|\zeta_2|)^2-|v|^2+ w -2\bar \zeta_2 v u_0 + v_2 \\ 
			\sqrt{2}(|q_0|+|\zeta_2|+v)u_0 \\ 1 \end{matrix}\right).
	\end{eqnarray*}
	The images of $F^{(\infty)}_{|q_0|}$ and $F^{(\infty)}_{|q_0| + |\zeta_2|}$ under the vertical projection $\Pi$ are the hyperplanes in $\mathbb{H}$
	$$
	\Re\left(q ~\frac{\bar \zeta_2}{|\zeta_2|}\right)= |q_0|, \quad \Re\left(q ~\frac{\bar \zeta_2}{|\zeta_2|}\right)= |q_0| + |\zeta_2| .
	$$
	It follows that the strip bounded by these two fans is a fundamental region for $\langle A\rangle$. The image under the vertical projection
	is the strip 
	$$
	S_A(|q_0|)=\left\{q \in \mathbb{H}\ :\ 
	|q_0|\le  \Re\left(q ~\frac{\bar \zeta_2}{|\zeta_2|}\right) \le |q_0|+|\zeta_2| \right\}.
	$$
	Therefore, if there exists a fundamental domain $D_B$ for  $\langle B\rangle$ that contains the complement of this strip, then the Klein combination theorem implies that the subgroup $\langle A, B\rangle$ is discrete and free. In particular, this condition is satisfied if the vertical projection of $\partial D_B$ is contained in $S_A$, and the vertical projection of $D_B$ contains the complement of $S_A(|q_0|)$.

	To see this, consider a point in the complement of $S_A(|q_0|)$. By construction, its pre-image under the vertical projection contains at least one point of $D_B$  and no points on its boundary. Since the pre-image is connected and path-connected (as the fibers of the vertical projection are copies of $\mathbb{R}^3$), all points in this pre-image must be contained in $D_B$, as required.
	
	\begin{proposition}\label{prop-fan-domain}
		Let $p_i=(\zeta_i,v_i)\in\HQ$ be distinct points and let $A$ and $B$ be the 
		associated Heisenberg translations fixing $\infty$ and $o$ respectively, as given by \eqref{eq-A-B}. Suppose that
		\begin{eqnarray*}
			|\zeta_2| K(p_1) 
			&\ge& 
			\frac{2 |\zeta_1|^{2}\Re (\zeta_1 
				\overline{\zeta_2})}{|\zeta_2|K^3(p_1)} 
			+ 2.
		\end{eqnarray*}
		Then the vertical projection of the isometric spheres of $B$ and $B^{-1}$ is contained in the strip $S_A(|q_0|)$ for some fixed $q_0=|q_0|\frac{\overline{\zeta_2}}{|\zeta_2|}$.
	\end{proposition}
	\begin{proof}
		The vertical projection of the isometric sphere associated with $B$ is a sphere in $\mathbb{H}$. The radius of this sphere is given by $r_B= 1/K^2(p_1),$ whose center is located at $\widetilde{p_1} = -\zeta_1(\kappa(p_1^{-1}))^{-1}.$ Any point $q$ on the vertical projection of the isometric sphere of $B$ can be expressed as $q= \widetilde{p_1} + r_B ~u,$ where $u$ is a unit quaternion.
		Therefore, $$\Re\left(q \frac{\overline{ \zeta_2}}{|\zeta_2|}\right) = \Re \left(\widetilde{p_1} \frac{\overline{\zeta_2}}{|\zeta_2|}\right) + r_B\Re\left(u\frac{\overline{ \zeta_2}}{|\zeta_2|}\right).$$
		Similarly, the vertical projection of the isometric sphere associated with $B^{-1}$ is also a sphere, with radius $r_{B^{-1}}=r_B= 1/K^2(p_1)$ and center ${\widetilde p_2}= \zeta_1(\kappa(p_1))^{-1}.$ Hence, any point $q$ on the vertical projection of the isometric sphere of $B^{-1}$ can be written as $q= {\widetilde p_2} + r_{B^{-1}}u,$ where $u$ is a unit quaternion. 
		Thus we have 
		$$
		\Re\left(q \frac{\overline{\zeta_2}}{|\zeta_2|}\right) = \Re \left({\widetilde p_2} \frac{\overline{\zeta_2}}{|\zeta_2|}\right) + r_{B^{-1}}\Re\left(u~ \frac{\overline{\zeta_2}}{|\zeta_2|}\right).
		$$
		Since $u\frac{\overline{\zeta_2}}{|\zeta_2|}\in \Sp(1),$ we have  $-1 \leq \Re\left(u \frac{\overline{\zeta_2}}{|\zeta_2|}\right)\leq 1.$ Therefore, for all points $q$ in the vertical projections of the isometric spheres
		of $B$ and $B^{-1}$, we have
		$$
		\Re \left({\widetilde p_2}\frac{\overline{\zeta_2}}{|\zeta_2|}\right) - r_{B^{-1}} \leq \Re\left(q\frac{\overline{\zeta_2}}{|\zeta_2|}\right) \leq 
		\Re \left({\widetilde p_1} \frac{\overline{\zeta_2}}{|\zeta_2|}\right) + r_{B}.
		$$
		Hence, these vertical projections lie in a strip of the form $S_A(|q_0|),$ provided the following inequality holds
		$$
		\Re \left({\widetilde p_1}\frac{\overline{} \zeta_2}{|\zeta_2|}\right)- \Re \left({\widetilde p_2}\frac{\overline{\zeta_2}}{|\zeta_2|}\right)+ 2r_{B}\leq|\zeta_2|,
		$$
		which simplifies to 
		$$ 
		\Re \left(({\widetilde p_1}-{\widetilde p_2}) \frac{\overline{\zeta_2}}{|\zeta_2|}\right) + 2r_{B} \leq|\zeta_2|.
		$$
		Equivalently,
		$$ 
		\Re \left[\left(  -\zeta_1(\kappa(p_1^{-1}))^{-1}-\zeta_1(\kappa(p_1))^{-1}\right) \frac{\overline{\zeta_2}}{|\zeta_2|}\right] + 2r_{B} \leq|\zeta_2|.
		$$
		Since $r_B=1/K(p_1)$, multiplying both sides by $K(p_1)$, we obtain
		\begin{eqnarray*}
			|\zeta_2| K(p_1) 
			&\ge& 
			\frac{K(p_1)}{|\zeta_2|} \cdot 
			\Re \left[-\zeta_1(\kappa(p_1^{-1}))^{-1}-\left(\kappa(p_1))^{-1}\right)
			{\overline \zeta_2} 
			\right] 
			+ 2\\
			&\ge& 
			\frac{2 |\zeta_1|^{2}\Re (\zeta_1 
				\bar{\zeta}_2 )}{|\zeta_2|K^3(p_1)} 
			+ 2,
		\end{eqnarray*}
		from which we obtain condition (\ref{main-2}) of Theorem~\ref{thm-main-improved}.
	\end{proof}
	By swapping roles of $A$ and $B$, we also obtain condition (\ref{main-3}) of Theorem~\ref{thm-main-improved}.
	\section{Appendix}\label{sec-appendix}
	
	\noindent{\it Proof of Lemma \ref{lem-trig}}.
	Suppose first that $\alpha=\pm\pi/4$. Then $f_\alpha(\theta)=(1/\sqrt{2})(\cos\theta\mp\sin\theta).$ In this case, the function attains its maximum value of at $\theta=\mp\pi/4$, and the claim holds.
	Next, suppose that $\alpha\in(-\pi/4,\pi/4)$ and that $\theta \in (-\pi/4,\pi/4)$. Differentiating $f_\alpha(\theta)$ with respect to $\theta$ gives
	$$
	f'_\alpha(\theta)=-\sin(\theta+\alpha)-\frac{\sin(2\theta)\sqrt{\cos(2\alpha)}}{\sqrt{\cos(2\theta)}}.
	$$
	Note that if $f'_\alpha(\theta)=0$ then $\sin(\theta+\alpha)$ and $\sin(2\theta)$ have opposite signs. Hence, either
	$\alpha<-\theta$ or $\alpha>-\theta$. 
	Now, if $f'_\alpha(\theta)=0$ then
	$$
	\sin^2(\theta+\alpha)\cos(2\theta) - \sin^2(2\theta)\cos(2\alpha)=0.
	$$
	By using standard trigonometric identities,
	and after simplification, the above equation may be written as follows:
	\begin{eqnarray*}
		0
		& = & \sin(\alpha-\theta) \left(\cos\alpha\left(3\cos^2\theta\sin\theta+\sin^3\theta\right)
		+\sin\alpha\left(\cos^3\theta+3\cos\theta\sin^2\theta\right)\right).
	\end{eqnarray*}
	Since $\sin(\alpha-\theta)\neq 0$, we deduce that
	\begin{eqnarray*}
		0 & = & \cos\alpha\Bigl(\bigl(\cos\theta+\sin\theta\bigr)^3-\bigl(\cos\theta-\sin\theta\bigr)^3\Bigr) \\
		&& \quad +\sin\alpha\Bigl(\bigl(\cos\theta+\sin\theta\bigr)^3+\bigl(\cos\theta-\sin\theta\bigr)^3\Bigr).
	\end{eqnarray*}
	Hence, the vanishing of the derivative implies that at critical points we have
	$$
	\frac{\cos\theta+\sin\theta}{\cos\theta-\sin\theta}
	=\left(\frac{\cos\alpha-\sin\alpha}{\cos\alpha+\sin\alpha}\right)^{1/3}.
	$$
	Let
	$
	\lambda=\bigl(\cos\alpha-\sin\alpha\bigr)^{1/3},\quad
	\mu =\bigl(\cos\alpha+\sin\alpha\bigr)^{1/3}.
	$
	Then
	$$
	\cos\alpha=\frac{\lambda^3+\mu^3}{2},\quad \sin\alpha=\frac{-\lambda^3+\mu^3}{2}.
	$$
	Define $\theta_0\in(-\alpha,0)$ such that 
	$$
	\cos\theta_0=\frac{\lambda+\mu}{\sqrt{2(\lambda^2+\mu^2)}},\quad 
	\sin\theta_0=\frac{\lambda-\mu}{\sqrt{2(\lambda^2+\mu^2))}}.
	$$
	Substituting this, we see that 
	$$
	\cos(\theta_0+\alpha) = \frac{\lambda^4+\mu^4}{\sqrt{2(\lambda^2+\mu^2))}}, \quad
	\cos(2\theta_0)  =  \frac{4\lambda\mu}{2(\lambda^2+\mu^2)}, \quad
	\cos(2\alpha) =  \lambda^3\mu^3.
	$$
	Therefore,
	\begin{eqnarray*}
		f_\alpha(\theta_0) & = & \frac{\lambda^4+\mu^4}{\sqrt{2(\lambda^2+\mu^2))}}+
		\frac{2\lambda^2\mu^2}{\sqrt{2(\lambda^2+\mu^2))}} \\
		& = & \frac{(\lambda^2+\mu^2)^{3/2}}{\sqrt{2}} \\
		& = & \frac{1}{\sqrt{2}}\Bigl(\bigl(\cos(\alpha)-\sin(\alpha)\bigr)^{2/3}+\bigl(\cos(\alpha)+\sin(\alpha)\bigr)^{2/3}\Bigr)^{3/2} \\
		& = & \frac{1}{\sqrt{2}}\Bigl(\bigl(1-\sin(2\alpha)\bigr)^{1/3}+\bigl(1+\sin(2\alpha)\bigr)^{1/3}\Bigr)^{3/2}.
	\end{eqnarray*}
	We now prove that the critical point $\theta_0$ corresponds to a global maximum of $f_\alpha$. The second derivative is
	$$
	f_\alpha''(\theta)=-\cos(\theta+\alpha)-2\sqrt{\cos(2\theta)\cos(2\alpha)}-\frac{\sin^2(2\theta)\sqrt{\cos(2\alpha)}}{(\cos(2\theta))^{3/2}}.
	$$
	Since
	$$
	\sin(2\theta_0)=\frac{\lambda^2-\mu^2}{\lambda^2+\mu^2},
	$$
	we obtain
	$$
	f_\alpha''(\theta_0)=-\frac{3\sqrt{2}}{4}(\lambda^2+\mu^2)^{3/2}<0,
	$$
	and this proves that we have a local maximum at $\theta_0$ and
	$$
	f_\alpha(\theta_0)=\frac{1}{\sqrt{2}}\Bigl(\bigl(1-\sin(2\alpha)\bigr)^{1/3}+\bigl(1+\sin(2\alpha)\bigr)^{1/3}\Bigr)^{3/2}.
	$$
	To confirm that it is a global maximum, it suffices to check that
	$$f_\alpha( \pi/4)= \cos( \pi/4 +\alpha) \leq \frac{1}{\sqrt{2}}\Bigl(\bigl(1-\sin(2\alpha)\bigr)^{1/3}+\bigl(1+\sin(2\alpha)\bigr)^{1/3}\Bigr)^{3/2},$$ 
	$$f_\alpha( -\pi/4)= \cos( -\pi/4 +\alpha) \leq \frac{1}{\sqrt{2}}\Bigl(\bigl(1-\sin(2\alpha)\bigr)^{1/3}+\bigl(1+\sin(2\alpha)\bigr)^{1/3}\Bigr)^{3/2}.$$  
	From the symmetry in the variables $\alpha $ and $\theta$, we have
	$\cos(\pi/4+\alpha)\le f_\alpha(\theta_0)$ for all $\alpha \in(-\pi/4,\pi/4)$
	and equality holds if and only if $\alpha= -\pi/4$. Similarly, we also have that $\cos(\pi/4-\alpha)\le f_\alpha(\theta_0)$ for all $\alpha \in(-\pi/4,\pi/4)$, with equality holds if and only if
	$\alpha= \pi/4.$  This completes the proof.
	\qed
	
	\bigskip
	
	\noindent{\it Proof of Lemma \ref{mam-bound}.}
	First, we observe that the values of the function $h$ at the boundary points are $h(\pm \frac{\pi}{2}) = 2.$ Next, assume that $\psi_1 \in \left(-\frac{\pi}{2}, \frac{\pi}{2}\right)$, and define the function
	$$
	g(\psi_1)=\left(1 - \sin \psi_1 \right)^{1/3} + \left(1 + \sin \psi_1 \right)^{1/3}
	$$
	on the open interval $I=(-\pi/2,\,\pi/2)$, so that $h(\psi_1)= \sqrt{2}.(g(\psi_1))^{1/3}$. 
	Then
	\begin{eqnarray*}
		g'(\psi_1)&=&(1/3)\cos\psi_1\left( (1+\sin\psi_1)^{-2/3}-(1-\sin\psi_1)^{-2/3}\right),
	\end{eqnarray*}
	which vanishes in $I$ only when $\psi_1=0$. We now write
	$
	g'(\psi_1)=(1/3)\cos\psi_1\cdot f(\psi_1) ,
	$
	where $f({\psi_1})= (1+\sin\psi_1)^{-2/3}-(1-\sin\psi_1)^{-2/3}$. Differentiating again and evaluating at $\psi_1=0$, we get
	$
	g''(0)=1/3f'(0)=-4/9<0.
	$
	Therefore $g(\psi_1)\le g(0)=2$ for each $\psi_1\in I$ and hence $h(\psi_1)\le h(0)=4$ for each $\psi_1$ in $[-\pi/2,\pi/2]$. Thus, the maximum value of $h$ is $4$, attained at $\psi_1=0$ and the minimum is $2$, attained at $\psi_1=\pm \frac{\pi}{2}$, which completes the proof. 
	\qed
	
	\bigskip
	
	\noindent{\it Proof of Lemma \ref{lem-bound}.} 
	First we identify the critical points of this function on the boundary of the domain $A$. If $\phi=\pi$ then
	\begin{eqnarray*}
		f_1(\psi_2)= \eta_{\psi_1}(\psi_2, \pi) & = & 2~ \cos \left(\frac{\psi_1 + \psi_2}{2}\right)  + 2\sqrt{\cos\psi_2 \cos\psi_1 .} \\
	\end{eqnarray*}
	By Lemma \ref{lem-trig} we have $$ f_1(\psi_2) \leq \sqrt{2} \left( \left(1 - \sin \psi_1 \right)^{1/3} + \left(1 + \sin \psi_1 \right)^{1/3} \right)^{3/2}.$$ 
	When $\phi=0$ we get
	\begin{eqnarray*}
		f_2(\psi_2)=\eta_{\psi_1}(\psi_2, 0) & = & 2~ \cos \left(\frac{\psi_1 - \psi_2}{2}\right)  + 2\sqrt{\cos\psi_2 \cos\psi_1 } \\
	\end{eqnarray*}
	and by a similar argument we obtain $$f_2(\psi_2)\leq \sqrt{2} \left( \left(1 - \sin \psi_1 \right)^{1/3} + \left(1 + \sin \psi_1 \right)^{1/3} \right)^{3/2}.$$ 
	Now, when \(\psi_2 = \pm\frac{\pi}{2}\),
	\[
	g_\pm(\phi)=\eta_{\psi_1}\left(\pm\tfrac{\pi}{2}, \phi\right) = \sqrt{2\left(1  \pm\cos \phi \cdot \sin \psi_1\right)}\le 2,
	\]
	and since by Lemma \ref{mam-bound} inequality $2\leq \max_{\psi_2, \phi} \eta_{\psi_1}(\psi_2, \phi)$ holds,
	we eventually obtain
	\[
	g_\pm(\phi)  \leq 2 \leq \max_{\psi_2, \phi} \eta_{\psi_1}(\psi_2, \phi).
	\]
	In the interior of $A$, the partial derivative with respect to \(\phi\) is
	\[
	\frac{\partial \eta_{\psi_1}(\psi_2, \phi)}{\partial \phi} = \frac{-\sin \phi \sin \psi_1 \sin \psi_2}{\sqrt{2 \left( 1 + \cos \psi_1 \cos \psi_2 + \cos \phi \sin \psi_1 \sin \psi_2 \right)}}.
	\]
	This vanishes either when $\psi_1 = 0$, or $\psi_2=0$, or $\phi =0$, $\phi= \pi$. For the special case $\psi_1=0$,
	$
	\eta_0(\psi_2,\phi)=\sqrt{2(1+\cos\psi_2}+2\sqrt{\cos\psi_2}\le 2.
	$
	In any case, the non-boundary critical points of $\eta_{\psi_1}(\psi_2,\phi)$ appear only when $\psi_2=0$.
	Now, the partial derivative with respect to $\psi_2$ is
	\[
	\frac{\partial \eta_{\psi_1}}{\partial \psi_2} = 
	\frac{-\cos\psi_1 \sin\psi_2 - \cos\phi \sin\psi_1 \cos\psi_2}
	{\sqrt{2 \left( 1 + \cos\psi_1 \cos\psi_2 + \cos\phi \sin\psi_1 \sin\psi_2 \right)}}
	- \frac{\sin\psi_2 \cos\psi_1}{\sqrt{\cos\psi_1 \cos\psi_2}}.
	\]
	Consider the equation 
	$
	\frac{\partial \eta_{\psi_1}}{\partial \psi_2} = 0.
	$
	When \(\psi_2 = 0\)  this becomes
	$$
	\frac{-\cos\phi \sin\psi_1}{\sqrt{2(1 + \cos\psi_1)}} = 0 \Rightarrow \cos\phi \sin\psi_1 = 0.
	$$
	For  \(\psi_1 \neq 0\) we thus have \(\phi = \frac{\pi}{2}\) and the only critical point of $\eta_{\psi_1}(\psi_2,\phi)$, $\psi_1\neq 0$, in the interior of $A$ is the point $(0,\pi/2)$.
	For the second derivatives at the critical point $(0,\pi/2)$ we observe first that
	$$
	\frac{\partial\eta_{\psi_1}}{\partial\phi}=\sin\psi_2\cdot f(\psi_1,\psi_2,\phi).
	$$
	Then the term $\sin\psi_2$ will remain as a multiplicative term after the derivation w.r.t. $\phi$ and thus
	$$
	\frac{\partial^2\eta_{\psi_1}}{\partial\phi^2}(0,\pi/2)=0.
	$$
	Now,
	$$
	\frac{\partial^2\eta_{\psi_1}}{\partial\psi_2\partial\phi}=\cos\psi_2\cdot f(\psi_1,\psi_2,\phi)+\sin\psi_2\cdot \frac{\partial f}{\partial \psi_2}.
	$$
	Therefore,
	$$
	\frac{\partial^2\eta_{\psi_1}}{\partial\psi_2\partial\phi}(0,\pi/2)=\sin\psi_1\cdot f(\psi_1,0,\pi/2)=-\frac{\sin\psi_1}{\sqrt{2(1+\cos\psi_1)}}.
	$$
	It turns out that the Hessian determinant of $\eta_{\psi_1}$ at $(0,\pi/2)$ is
	$$
	{\rm Hess}(\eta_{\psi_1}(0,\pi/2)=-\left(\frac{\partial^2\eta_{\psi_1}}{\partial\psi_2\partial\phi}(0,\pi/2)\right)^2=-\frac{\sin^2\psi_1}{2(1+\cos\psi_1)}<0,
	$$
	since $\psi_1\neq 0$. Therefore the point $(0,\pi/2)$ is a saddle point for $\eta_{\psi_1}$ and therefore the maximum is attained at the boundary.
	We conclude that the  maximum value of $\eta_{\psi_1}(\psi_2, \phi)$ is indeed
	$$\max_{\psi_2, \phi} \eta_{\psi_1}(\psi_2, \phi) =  \sqrt{2} \left( \left(1 - \sin \psi_1 \right)^{1/3} + \left(1 + \sin \psi_1 \right)^{1/3} \right)^{3/2}.$$
	\qed


\begin{thebibliography}{999}
		
		
		
		
		\bibitem[Gol99]{gol}
		W.M. Goldman, 
		\newblock \emph{Complex hyperbolic geometry}, 
		\newblock Oxford University Press, Oxford, New York, (1999).
		
		\bibitem[WP]{wp1} 
		Wensheng Cao and J.~R. Parker,
		\newblock \emph{Shimizu’s lemma for Quaternionic Hyperbolic Space},
		\newblock Computational Methods and Function Theory,
		\newblock {\bf 18} (2018), 159–191. 
		
		\bibitem[Ign79]{ig78}
		Yu.~A. Ignatov,
		\newblock \emph{Free and non-free subgroups of ${\rm PSL}_2(\mathbb C)$ generated by two parabolic elements},
		\newblock Math.\ of the USSR Sbornik,
		\newblock {\bf 35} (1979) 49--55.
		
		\bibitem[JL14]{JL} 
		Jaime Leonardo Orjuela Chamorro,
		\newblock \emph{Complete minimal hypersurfaces in quaternionic hyperbolic space},
		\newblock Geom. Dedicata
		\newblock {\bf 172} (2014),47–67. 
		
		
		\bibitem[KP]{kp1} 
		Inkang Kim and J.~R. Parker,
		\newblock \emph{Geometry of quaternionic hyperbolic manifolds},
		\newblock Math. Proc. Camb. Phil. Soc.
		\newblock {\bf 135(2)} (2003), 291 - 320. 
		
		
		
		\bibitem[KS94]{ks}
		L.~Keen and C.~Series,
		\newblock \emph{The Riley slice of Schottky space}, 
		\newblock Proc. London Math. Soc, 
		\newblock {\bf 69} (1994), 72--90.
		
		
		\bibitem[Par1]{par}
		John R.~Parker,
		\newblock \emph{On Ford isometric spheres in complex hyperbolic space}, 
		\newblock Math. Proc. Camb. Phil. Soc,
		\newblock {\bf 115} (1994), 501--512.
		
		\bibitem[LU69]{lyu} 
		R. C. Lyndon and J. L. Ullman,
		\newblock \emph{Groups generated by two parabolic linear fractional transformations},
		\newblock Canad. J. Math. 
		\newblock {\bf 21} (1969), 1388--1403.
		
		\bibitem[SD]{sd}
		Sagar~B. {Kalane} and Devendra. {Tiwari},
		\newblock \emph{On Free Groups Generated by Two Heisenberg Translations}, 
		\newblock Proc. Indian Acad. Sci. Math. Sci.,
		\newblock {\bf 132(38)} (2022), 1--10.
		
		\bibitem[GK1]{gk1}
		Sagar~B. {Kalane} and John R. {Parker},
		\newblock \emph{Free groups generated by two Parabolic Maps},
		\newblock Mathematische Zeitschrift, 
		\newblock 303 (2023), 1-23.
		
		
		
	\end{thebibliography}
\end{document}